\documentclass[11pt]{amsart}
 \usepackage[dvips]{epsfig}
\usepackage{graphicx}
\usepackage{latexsym}
\usepackage{amsfonts,amsmath,amssymb}
\newtheorem{theorem}{Theorem}[section]
\newtheorem{proposition}[theorem]{Proposition}
\newtheorem{lemma}[theorem]{Lemma}
\newtheorem{definition}[theorem]{Definition}
\newtheorem{corollary}[theorem]{Corollary}
\newtheorem{example}[theorem]{Example}

\newtheorem{conjecture}[theorem]{Conjecture}

\author[M. B\'ona]{Mikl\'os  B\'ona}
\title[Bound for Permutations Avoiding a Pattern of  Given Length]{On the Best
 Upper Bound for Permutations Avoiding a Pattern of
a Given Length}
\address{\rm M. B\'ona, Department of Mathematics, 
University of Florida,
358 Little Hall, 
PO Box 118105, 
Gainesville, FL 32611--8105 (USA)
}

\date{\today}

\begin{document}

\begin{abstract}
Numerical evidence suggests that certain permutation patterns of length $k$
are easier to avoid than any other patterns of that same length. We prove
that these patterns are avoided by no more than $(2.25k^2)^n$ permutations
of length $n$. In light of this, we conjecture that no pattern of length $k$
is avoided by more than that many permutations of length $n$. 
\end{abstract}

\maketitle

\section{Introduction}
\subsection{Upper Bounds for Pattern Avoiding Permutations}
 The theory of pattern avoiding permutations has
seen tremendous progress during the last two decades. The key definition is
the following. Let $k\leq n$, let $p=p_1p_2\cdots p_n$ be a permutation of
 length $n$, and
let $q=q_1q_2\cdots q_k$ be a permutation of length $k$. We say that $p$
 avoids $q$ if there 
are no $k$ indices $i_1<i_2<\cdots <i_k$ so that for all $a$ and $b$, 
the inequality $p_{i_a}<p_{i_b}$ holds if and only if the inequality $q_a<q_b$
holds. For instance, $p=2537164$ avoids $q=1234$ because $p$ does not contain
an increasing subsequence of length four. See \cite{combperm} for an overview
of the main results on pattern avoiding permutations.

Let $S_n(q)$ be the number of permutations of length $n$ (or, in what follows,
$n$-permutations) that avoid the pattern $q$. Since the spectacular result of
Adam Marcus and G\'abor Tardos \cite{marcus}, it is known that for 
every pattern $q$, there exists a constant $c_q$ so that the inequality
$S_n(q)\leq c_q^n$ holds for all $n$. As there are only $k!$ patterns of 
length $q$, it follows that for all positive integers $k$, there exists
a constant $c_k$ so that for all patterns $q$ of length $k$, the inequality
\begin{equation} \label{univcon} S_n(q)\leq c_k^n \end{equation}
holds for all positive integers $n$. 

However, the quest of finding the {\em best}
 constant $c_k$ is in (\ref{univcon}),
is wide open. The result of Marcus and Tardos \cite{marcus} has only shown that
$c_k\leq 15^{2k^4\cdot {k^2\choose k}}$. Josef Cibulka \cite{cibulka} has
 improved this
bound by showing that $c_k\leq 2^{O(k \log k)}$, but even this bound seems
to be very far from reality, as we will explain. 

 Richard Arratia \cite{arratia} has conjectured that 
$c_k=(k-1)^2$ is sufficient for all $k$, but this conjecture was refuted
by Albert and al \cite{albert}, who proved that if $n$ is large enough, then
$S_n(1324)\geq 9.42^n$.  

\subsection{Layered Patterns}
A {\em layered} pattern is a pattern consisting of decreasing subsequences
(the layers) so that the entries decrease within the layers but increase
among the layers, as in 3215476. Equivalently, a pattern is layered if and
only if it avoids both 231 and 312. Layered permutations are important since
numerical evidence (computed first by Julian West \cite{west} and later 
replicated by many others)
supports the following conjecture.

\begin{conjecture} \label{layered}
Let $q$ be a non-layered pattern of length $k$, and let $Q$ be a layered
pattern of length $k$. Then for all positive integers $n$, 
the inequality
\begin{equation} \label{ineqlayered} S_n(q)\leq S_n(Q)
\end{equation}
holds. 
\end{conjecture}
If Conjecture \ref{layered} holds, then any upper bound that we can prove
for all {\em layered}
 patterns of length $k$ is also an upper bound for {\em all}
patterns of length $k$. This has motivated several attempts
 to find a constant $CL_k$
so that $S_n(Q)\leq CL_k^n$ for all layered patterns $Q$ of length $k$. 
It follows from results in \cite{tenacious}, \cite{integer} and
\cite{records} that $CL_k\leq O(2^k)$. A much stronger, recent result
of Anders Claesson, Vit Jelinek and Einar Steingr\'{\i}msson \cite{claesson}
shows that $CL_k\leq 4k^2$ holds.

Let $M_{2m}$ denote the pattern $132\cdots (2m-1)(2m-2)2m$, and let 
$M_{2m-1}$ denote the pattern obtained from $M_{2m}$ by removing the first
entry and then relabeling. That is, $M_{2m-1}=21\cdots (2m-2)(2m-3)(2m-1)$.
So for instance, $M_3=213$, and $M_4=1324$, while $M_5=21435$, and $M_6=132546$.
A different version of Conjecture \ref{layered}, also supported by numerical
 evidence,  is the following. 
\begin{conjecture} \label{strongerconj}
Let $m\geq 2$. Then 
\begin{enumerate} \item[(A)]
for all positive integers $n$, and for all patterns $q$ of length $2m-1$, the
inequality 
\[ S_n(q)\leq S_n(M_{2m-1}) \] holds, and
\item[(B)] for all positive integers $n$, and for all patterns $q$ of length
 $2m$, the
inequality 
\[ S_n(q)\leq S_n(M_{2m}) \] holds.
\end{enumerate}
\end{conjecture}

In this paper, we will prove that if $k=2m-1$, then the inequality
 $S_n(M_{2m-1})\leq
(2.25k^2)^n$ holds, and if $k=2m$, then the inequality $S_n(M_{2m})\leq 
(2.25k^2)^n$ holds. This means that if Conjecture \ref{strongerconj} holds,
 then 
$c_k\leq 2.25k^2$. (In fact, we prove a slightly stronger upper bound for
 $c_k$.)
Note that $c_k\geq (k-1)^2$ has been known since Amitaj Regev's paper
 \cite{regev}. 

Our proof will be inductive, but even the initial case of our induction will
 depend on 
a result that has only been recently proved. 

\section{Composing and Decomposing Patterns}

The following useful definitions describe two simple but crucial ways in
 which  patterns can be composed.

\begin{definition}
Let $q$ be a pattern of length $k$ and let $t$ be a pattern of length $m$.
Then $q\oplus t$ is the pattern of length $k+m$ defined by
\[ (q\oplus t)_i= \left\{ \begin{array}{l@{\ }l}
q_i  \hbox{  if $i\leq k$},\\
\\
 t_{i-k} +k \hbox{  if $i>k$. }
\end{array}\right.
\]
\end{definition}

In other words, $q\oplus t$ is the concatenation of $q$ and $t$ so that 
all entries of $t$ are increased by the size of $q$.

\begin{example}
If $q=3142$ and $t=132$, then $q\oplus t=3142576$.
\end{example}

\begin{definition}
Let $q$ be a pattern of length $k$ and let $t$ be a pattern of length $m$.
Then $q\ominus t$ is the pattern of length $k+m$ defined by
\[ (q\ominus t)_i= \left\{ \begin{array}{l@{\ }l}
q_i + m \hbox{  if $i\leq k$},\\
\\
 t_{i-k} \hbox{  if $i>k$. }
\end{array}\right.
\]
\end{definition}

In other words, $q\ominus t$ is the concatenation of $q$ and $t$ so that 
all entries of $q$ are increased by the size of $t$.

\begin{example}
If $q=3142$ and $t=132$, then $q\ominus t=6475132$.
\end{example}

The following strong theorem of Claesson, Jelinek and Steingr\'{\i}msson
 describes an important way in which 
permutations avoiding a  long given pattern of a specific kind can be 
decomposed into two permutations, each of which avoids a shorter pattern.

\begin{theorem} \cite{claesson} \label{Claesson}
Let $\sigma$, $\tau$, and $\rho$ be three permutations.
Let $p$ be a permutation that avoids $\sigma \oplus (\tau \ominus 1) \oplus
\rho$. Then it is possible to color each entry of $p$ red or blue so that
the red entries of $p$ form a $\sigma \oplus (\tau \ominus 1)$-avoiding
permutation and the blue entries of $p$ form a $(\tau \ominus 1) \oplus
\rho$-avoiding permutation.
\end{theorem}

\begin{example} \label{easyex}
Let $\sigma$, $\tau$, and $\rho$ each be the one-element pattern 1. 
Then Theorem \ref{Claesson} says that it is possible to color the entries
of  a 1324-avoiding permutation red or blue so that the red entries
form a 132-avoiding permutation and the blue entries form a 213-avoiding
permutation. 
\end{example}

\begin{proof} (of Theorem \ref{Claesson})
Color the entries of $p$ one by one, going left to right, according to the
following three rules.
\begin{enumerate}
\item If coloring $p_i$ red creates a red copy of 
$\sigma \oplus (\tau \ominus 1)$, then color $p_i$ blue. 
\item If $p_i$ is larger than a blue entry on its left, then color
$p_i$ blue. 
\item Otherwise color $p_i$ red.
\end{enumerate}
It is then proved in \cite{claesson} that this coloring has the required
properties. 
\end{proof}

\begin{definition} The coloring defined in the preceding proof is called
the {\em canonical} coloring of $p$ (with respect to $\sigma$, $\tau$, and
$\rho$). 
\end{definition}

\begin{example} Let  $\sigma$, $\tau$, and $\rho$ each 
be the one-element pattern as
in Example \ref{easyex}, and let $p=3612745$. Then $p$ is a 1324-avoiding
 permutation.
In the canonical coloring of $p$ with respect to $\sigma$, $\tau$, and $\rho$, 
the red entries are 3, 6, 1, 2, and 7, while the blue entries are 4 and 5.
 It is easy
to verify that the string of red entries avoids 132, while the string of
 blue entries avoids
213. 
\end{example}

\section{An Inductive Argument}
In this section, we will present an inductive argument that shows how 
$M_{2m-1}$-avoiding and $M_{2m}$-avoiding permutations can be injectively
mapped into  pairs of certain words.  The precise statement will be
made in Lemma \ref{inducl}. Even the initial steps of this argument
are not obvious. The argument for $M_3$-avoiding permutations (that is, 
213-avoiding permutations), is not surprising. The argument for $M_4$-avoiding
permutations has only been recently found \cite{newbound}. These two arguments
are given in Section \ref{initial}, before the general result is announced 
in Section
\ref{inductive}.

\subsection{The Initial Steps} \label{initial}
Let $p=p_1p_2\cdots p_n$. We say that $p_i$ is a {\em right-to-left maximum}
in $p$ if it is larger than all entries on its right, that is, 
if $p_i>p_j$ for all $j>i$. We always consider $p_n$ a right-to-left maximum,
since the condition is vacuously true for that entry.

Let $V_2(n)$ be the set of all words of length $n$ over the alphabet $\{0,1\}$.
\begin{proposition} \label{catalan}
Let $p=p_1p_2\cdots p_n\in Av_n(213)$,  let $v(p_i)=0$ if $p_i$ is not a
right-to-left maxium, and let $v(p_i)=1$ if $p_i$ is a right-to-left
maximum. Set 
\[v(p)=v(p_1)v(p_2)\cdots v(p_n)\] and \[v'(p)= v(1)v(2)\cdots v(n).  \]
Then the map $f_3:Av_n(213)\rightarrow V_2(n)\times V_2(n)$ defined by
$f(p)=(v(p),v'(p))$ is injective. 
\end{proposition}

\begin{example} If $p=35412$, then $v(p)=01101$, and so $v'(p)=
01011$. Therefore, $f(35412)=(01101,01011)$.
\end{example} 

\begin{proof}
Let $(v,v')\in V_2(n)\times V_2(n)$, and let us assume that there is a
permutation $p\in Av_n(213)$ so that $f_3(p)=(v,v')$. Then the positions
of the 1s in $v$ reveal the positions in which $p$ must have right-to-left
maxima, and the positions of 1s in $v'$ reveal which entries of $p$ are
right-to-left maxima. It follows from the definition of right-to-left maxima
that the right-to-left maxima of $p$ must be in decreasing order from 
left to right. 

Once the right-to-left maxima of $p$ are in place, there is only one 
way to insert the remaining entries in the remaining slots. Indeed, 
going from right to left, in each step we must place the largest eligible
remaining entry (that is, the largest one whose insertion does not change
the set of right-to-left maxima). Indeed, if at some point during this procedure
we placed an entry $x$ instead of the eligible entry $y>x$, then 
$y$, $x$, and the closest right-to-left minimum on the right of $x$ would form
a 213-pattern.
\end{proof}

Note that by trivial arguments based on symmetry, analogous results
can be proved for 312-avoiding, 231-avoiding, and 132-avoiding permutations.
(We will actually use that last class.)

Also note that $v'(p)$ is nothing but the letters of the word $v(p)$ rearranged according
to the {\em inverse} $p^{-1}$ of $p$. That is, if $p^{-1}=p_{a_1}p_{a_2}\cdots p_{a_n}$, 
then $v'(p)=v(p^{-1})=v(p_{a_1})v(p_{a_2})\cdots v(p_{a_n})$.  

We say that a word $w$ over a finite alphabet has an {\em $XY$-factor} if there
is a letter $X$ in $w$ that is immediately followed by a letter $Y$. For instance,
the word 00011120 has no 02-factors. Let $p=p_1p_2\cdots p_n$. We say that $p_i$
is a {\em left-to-right minimum} if $p$ if $p_i<p_j$ for all $j<i$. In other words, a left-to-right
minimum is an entry that is less than everything on its left. 

Let $W_4(n)$ be the set of all words of length $n$ over the
alphabet $\{1,2,3,4\}$ that contain no 32-factors.  The following is a recent result of 
the present author. 

\begin{lemma} \cite{newbound} \label{newbound}
 Let $p=p_1p_2\cdots p_n\in Av_n(1324)$. Consider the 
canonical decomposition of $p$ into a 132-avoiding permutation of
red entries and a 213-avoiding permutation of blue entries as given
in Theorem \ref{Claesson}. Furthermore, define the word $w(p)$ as follows.
\begin{itemize}
\item If $p_i$ is a red entry that is a left-to-right minimum in the string
of red entries, let $w(p_i)=1$,
\item if $p_i$ is a red entry that is not a left-to-right minimum in the string
of red entries, let $w(p_i)=2$,
\item if $p_i$ is a blue entry that is not a right-to-left maximum in the string
of red entries, let $w(p_i)=3$, and
\item if $p_i$ is a blue entry that is a right-to-left maximum in the string
of blue entries, let $w(p_i)=4$.
\end{itemize} 
 Set 
\[w(p)=w(p_1)w(p_2)\cdots w(p_n)\] and \[w'(p)= w(1)w(2)\cdots w(n).  \]
Then the map $f_4:Av_n(1324)\rightarrow W_4(n)\times W_4(n)$ defined by
$f_4(p)=(w(p),w'(p))$ is injective. 
\end{lemma}

\begin{example} If $p=3612745$, then we get $w(p)=1212234$, and $w'(p)=1213424$, 
and we can easily see that neither $w(p)$ nor $w'(p)$ contains a 32-factor. 
\end{example}

\begin{proof} (of Lemma \ref{newbound})
Let $f_4(p)=(w(p),w'(p))$, and let us assume that $w(p)$ contains a 
32-factor, that is, there exists an index $i$ so that $w(p_i)=3$ and
$w(p_{i+1})=2$. That means that in particular, $p_i$ is blue and 
$p_{i+1}$ is red, so by the second rule of canonical coloring (as 
given in Theorem \ref{Claesson}), $p_i>p_{i+1}$. As $p_{i+1}$ is not
a left-to-right minimum, there is an entry $p_j$ with $j<i$ so that
$p_j<p_{i+1}$. Similarly, as $p_i$ is not a right-to-left maximum, 
there is an entry $p_\ell$ with $\ell>i+1$ so that $p_\ell >p_i$.
However, that means that $p_jp_ip_{i+1}p_\ell$ is a 1324-pattern, which is
a contradiction. An analogous argument (see \cite{newbound})
 shows that $w'(p)$ also avoids 1324. So $f_4$ indeed maps into
$W_4(n)\times W_4(n)$.

In order to see that $f_4$ is injective, we proceed in a way that is similar
to (but slightly more complex than) the way in which we proceeded in the
proof of Proposition \ref{catalan}. Let $(w,w')\in W_4(n)\times W_4(n)$,
and let us assume that there exists $p\in Av_n(1324)$ so that $f_4(p)=(w,w')$.
Then the positions of 3s and 4s in $w$ reveal {\em where} the blue entries
of $p$ are, and the positions of 3s and 4s in $w'$ reveal {\em what}
the blue entries of $p$ are. By Proposition \ref{catalan}, this is sufficient
information to recover the entire string of blue entries, since that string
is a 213-avoiding permutation. A dual argument works for the string of red
entries, since the red entries form a 132-avoiding permutation.
\end{proof}

\subsection{The Induction Step} \label{inductive}
In this part of our proof, we will often obtain an encoding of a long
permutation by partitioning it into two parts, encoding each part by 
disjoint alphabets, then combining the two images into one. The following
definition makes this concept more precise. If $s$ is a substring of the 
permutation $p$, let $|s|$ denote the length (number of entries) of $s$. 

\begin{definition} \label{merge}
Let $p=p_1p_2\cdots p_n$ be a permutation, and let $p'$ and $p''$ be two  
substrings
of $p$ so that each entry of $p$ belongs to exactly one of $p'$ and $p''$.

Let us assume that $a(p')$ and $a'(p')$ are words of length $|p'|$ over a
finite alphabet $A$, and $b(p'')$ and $b'(p'')$ are words of length 
$|p''|$ over a finite alphabet $B$ that is disjoint from $A$.

 Then the {\em merge} of 
$a(p')$ and $b(p'')$ is the word $w(p)=w_1w_2\cdots w_n$ of length $n$ 
over the finite alphabet $A\cup B$ whose $i$th letter $w_i$ is obtained as 
follows. 
\begin{enumerate}
\item If $p_i$ is 
the $r$th letter $p'$, then $w_i$ is equal to the $r$th letter
 of $a(p')$, and 
\item if $p_i$ is the $r$th letter of $p''$, then $w_i$ is 
the $r$th letter of $b(p'')$. 
\end{enumerate}

Furthermore, the merge of $a'(p')$ and $b'(p'')$ is the word
$w'(p)=w_1'w_2'\cdots w_n'$ obtained as follows.
\begin{enumerate}
\item If the entry $i$ of $p$ is the $t$th smallest entry 
 of $p'$, then $w_i'$ is equal to 
the $t$th letter of $a'(p')$, and,
\item if the entry $i$ of $p$ is the $t$th smallest entry 
 of $p''$, then $w_i'$ is equal to 
the $t$th letter of $b'(p'')$.
\end{enumerate}
\end{definition}

\begin{example}
Let $p=178942365$, let $p'=12$, and let $p''=7894365$. 
Let $a(p')=11$, and let $a(p'')=11$.
Furthermore, let $b(p'')=2222233$, and let $b'(p'')=2233222$.

Then we have $w(p)=122 221 233$ and $w'(p)=112233222$.
\end{example}

The following definition extends the notion of merges from words to functions in a 
natural way. 

\begin{definition} \label{mergef}
Let $p$, $p'$ and $p''$ be as in Definition \ref{merge}, and let $f(p')=(v_1,v_2)$, and
$g(p'')=(w_1,w_2)$, where the $v_i$ are words over the finite alphabet $A$, and the
$w_i$ are words over the finite alphabet $B$ that is disjoint from $A$. Then we say
that the function $h$ is the merge of $f$ and $g$ if $h(p)=(z_1,z_2)$, where $z_1$
is the merge of $v_1$ and $w_1$, and $z_2$ is the merge of $v_2$ and $w_2$.  
\end{definition}

\begin{example} \label{mergefex}
Let $p=687912435$, let $p'=612$, and let $p''=879435$. Let $f(p')=(000,000)$, and
let $g(p'')=112112$. If $h$ is the merge of $f$ and $g$, then $h(p)=(011200112,
001120112)$.
\end{example}

Recall that $M_{2m}$ denotes the pattern $132\cdots (2m-1)(2m-2)2m$, and 
$M_{2m-1}$ denotes the pattern obtained from $M_{2m}$ by removing the first
entry and then relabeling. That is, $M_{2m-1}=2143\cdots (2m-2)(2m-3)(2m-1)$.
So $M_4=1324$, while $M_5=21435$, and $M_6=132546$. 

In order to make the statement and proof of the following lemma easier to follow, 
we make the following general remark about the {\em indices} used in the lemma. 
The lemma will describe injections from certain sets of $q$-avoiding permutations into 
sets of pairs of certain words. These injections will be denoted by $f_k$, where $k$
is the length of $q$.  The co-domains
of the injections $f_k$ will be denoted by $V_a$ or $W_a$, where $a$ denotes the length of
the words in the co-domain of the $f_k$.

\begin{lemma} \label{inducl}
 Let $m\geq 2$. Let $Av_n(M_{t})$ denote the set of all 
$M_{t}$-avoiding $n$-permutations.

\begin{enumerate} \item[(a)] Let $V_{3m-4}(n)$ denote
the set of all words of length $n$ over the alphabet
 $\{0,1,2,\cdots ,3m-5\}$ that do not have any $(3i)(3i-1)$-factors for any
$i\geq 1$. 

 Then there is an injection \[f_{2m-1}:Av_n(M_{2m-1})\rightarrow
 V_{3m-4}(n)\times V_{3m-4}(n).\] 

\item[(b)]  Let $W_{3m-2}(n)$ denote
the set of all words of length $n$ over the alphabet
 $\{1,2,\cdots ,3m-2\}$ that do not have any $(3i)(3i-1)$-factors for any
$i\geq 1$. 

 Then there is an injection \[f_{2m}:Av_n(M_{2m})\rightarrow
 W_{3m-2}(n)\times W_{3m-2}(n).\] 
\end{enumerate}
\end{lemma}

\begin{proof} 
We prove the statements by induction on $m$. For $m=2$, the statements
are true. Indeed, for $m=2$, statement (a) is just the statement of Proposition \ref{catalan}, 
and statement (b) is just the statement of Lemma \ref{newbound}. 

Now let us assume that the statements are true for $m$, and let us prove them
for $m+1$. 
\begin{enumerate}
\item[(a)] First, we prove statement (a).
 Let $p\in Av_{n}(M_{2m+1})$. Color all entries of $p$ that are the leftmost
entry of an $M_{2m}$-pattern in $p$ green, and color all other entries of
$p$ yellow. Then, by definition, the string of all yellow entries of $p$
forms an $M_{2m}$-avoiding permutation $p''$. By part (b) of our induction
 hypothesis, the map $f_{2m}$ injectively maps this permutation
$p''$ into a pair of words  $(w_1,w_2)\in W_{3m-2}(|p''|)\times W_{3m-2}(|p''|)$. 
In order to define the image $f_{3m-1}(p)\in V_{3m-1} \times V_{3m-1}$, simply 
mark all green entries of $p$ by the letter 0. Let $p'$ be the string of
all green entries, and, to keep consistency with Definition \ref{merge},
let $g(p')=(00\cdots 0, 00\cdots 0)$, where both strings of 0s are of length
$|p'|$. Then we define $f_{2m+1}(p)$ as the merge of $g(p')$ and $f_{2m}(p')$
as defined in Definition \ref{mergef}.

\begin{example} For $p=687912435$, the reader is invited to revisit Example
\ref{mergefex}. With our current terminology, $p'$ is the string of green entries, 
$p''$ is the string of yellow entries, and $h=f_5$.
\end{example}

It is clear that $f_{2m+1}(p)$ indeed does not have a $(3i)(3i-1)$-factor
for any positive integer $i$, since positive integers correspond
to yellow entries of $p$, and the string of yellow entries avoids all
such factors by the induction hypothesis. 

In order to show that  the map $f_{2m+1}$ is injective, let us assume that
$(v_1,v_2)\in  V_{3m-1}(n)\times V_{3m-1}(n)$  equals $f_{2m+1}(p)$ for
some $p$. Then the positions of the yellow entries of $p$ are easy to 
recover, since these are the positions in which $v_1$ has a positive value. 
Similarly, the values of the yellow entries can be recovered as the positions
in which $v_2$ has a positive entry. Once the place and values of the yellow
entries of $p$ are found, the order in which these yellow entries is unique
since the map $f_{2m}$ that is applied to the string of yellow entries 
is injective. So the injective property of $f_{2m}$ will be proved if we
can show that there is only one way to place the green entries into the
remaining slots.

Let us fill the remaining slots with the green entries going right to left.
We claim that in each step, we must insert the largest remaining green entry
that is eligible to go into the given position (that is, that will start
an $M_{2m}$-pattern if inserted there). Indeed, let us assume that in a
given position, we do not proceed as described. That is, both $x$ and $y$
are eligible to be inserted in a given position $P$, but we insert $x$, even if
$x<y$. The fact that both $x$ and $y$ are eligible to be inserted in $P$
means that they both will be the first entry of an $M_{2m}$-pattern if inserted
in $P$. Let these patterns be $xM$ and $yM'$. Then we will create an 
$M_{2m+1}$ pattern, namely the pattern $yxM'$ when we eventually insert $y$ 
somewhere on the left of $x$.  

\item[(b)] Now we prove statement (b). Let $p\in Av_{n}(M_{2m+2})$.
Then Theorem \ref{Claesson} (with  $\sigma =1$, $\tau=1$ and
$\rho =M_{2m-1}$)
 shows that it is possible
to color the entries of $p$ red or blue so that the red entries form a
132-avoiding permutation and the blue entries form an $M_{2m+1}$-avoiding
permutation. Let us consider the canonical  coloring that achieves this and is 
given  in the proof of Theorem \cite{claesson}. 

Now we encode the string $p'$ of all red entries of $p$ in a manner that is 
analogous to what we saw in Proposition \ref{catalan}. We can do so, since the
$p'$ is a 132-avoiding permutation. To be more precise, 
mark each entry of $p'$ that is a left-to-right minimum in $p'$ by the letter $1$. Mark all remaining letters of
$p'$ by the letter $2$.  Define the words $a(p')$ and $a'(p')$ as in Proposition \ref{catalan}.
That is, let $p'=p'_1p'_2\cdots p'_r$, and let $v(p'_i)=1$ if $p'_i$ is a left-to-right minimum
in $p'$, and let $v(p'_i)=2$ otherwise. Then set $a(p')=v(p'_1)v(p'_2)\cdots v(p'_r)$, and 
set $a'(p')=v(p'_{j_1})v(p'_{j_2})\cdots v(p'_{j_r})$, where $p_{j_1}<p_{j_2}<\cdots <p_{j_r}$.

The string $p''$ of blue entries of $p$ forms an $M_{2m+1}$-avoiding permutation,
 so as we have just
seen in the proof of  statement (a), the string $p''$ can be injectively mapped into
a pair of words $(b(p''),b'(p'') )\in V_{3m-1} \times V_{3m-1}$ by the function
$f_{2m+1}$. Shift these letters by three, that is, turn each letter $i$ into 
a letter $i+3$ for all $i$ in $b(p'')$ and $b'(p'')$. Finally, define
$f_{2m+2}(p)=(w,w')$, where $w$ is the merge of $a(p')$ and $b(p'')$, 
and $w'$ is the merge of $a'(p')$ and $b'(p'')$. 

It is clear that $f_{2m+2}(p)=(w,w')$
 is a pair of words of length $n$ over the
alphabet $\{1,2,\cdots ,n\}$. It directly follows from the induction hypothesis
that neither $w$ nor $w'$ can contain a $(3i)(3i-1)$-factor for $i>1$.
There remains to show that neither $w$ nor $w'$ can contain a 32-factor.
In order to see this, let us assume that $w$ contains a 32-factor in
its $j$th and $(j+1)$st positions. The {\em type} of an entry of a permutation is
just the letter it is mapped into by $f_{2m+2}$.
 That means that in particular, $p_{j}$ is blue and $p_{j+1}$ is red, 
so, by the second rule of canonical colorings (see the proof of Theorem
\ref{Claesson}), $p_{j}>p_{j+1}$, since a blue entry cannot be followed by
a larger red entry. As $p_{j+1}$ is of type 2, it is not a
left-to-right minimum, so there exists
an index $d<j$  so that $p_d<p_{j+1}$. As $p_j$ is of type 3, it is of type
0 in $p''$, so there is an $M_{2n}$-pattern $p_jP$ in $p''$, and so, in $p$, 
whose first entry is $p_{j}$. However, that means that $p_dp_jp_{j+1}P$ is an
$M_{2m+2}$-pattern in $p$, which is a contradiction. So $w$ cannot contain
a 32-factor, and in an analogous way, nor can $w'$. 

Finally, we must show that $f_{2m+2}$ is injective. By now, the method we
used should not come as a surprise. Given a pair of words $(w,w')\in V_{3m+1}(n)
 \times V_{3m+1}(n)$,
we can recover the set and positions of the red entries, and the set
of positions of the blue entries of $p$, since the red entries are the ones
that are of type 1 or 2. After this, it follows from Proposition \ref{catalan}
that we can recover the string of the red entries, and it follows from
part (a) of this Lemma that we can recover the string of the blue entries.
\end{enumerate}
\end{proof}

\section{Computing the Upper bounds}

All there is left to do in order to find upper bounds on the numbers $S_n(M_{2m})$ 
and $S_n(M_{2m-1})$ is to find upper bounds on the sizes of the sets into which
the relevant permutations can be injectively mapped. It would be straightforward to 
simply find an upper bound on the {\em exponential growth rate} of these sequences, but 
we will carry out the slightly more cumbersome (but conceptually not difficult) task
of finding upper bounds for the sequence in the sense we described in the introduction.
 
\begin{proposition} \label{upperw} For all integers $m\geq 2$, 
we have
\[|W_{3m-2}(n)| =C_1(m)  \cdot \beta_1^n+C_2(m)\beta_2^n ,\]
where $\beta_1 = \frac{3m-2+\sqrt{9m^2-16m+8}}{2}$ and $\beta_2= 
\frac{3m-2-\sqrt{9m^2-16m+8}}{2}$, while 
 $C_1=\frac{\beta_1}{\beta_1-\beta_2}$  and 
$C_2=\frac{\beta_2}{\beta_2-\beta_1}$. 
\end{proposition}

\begin{proof} 
Let $b_0=1$, and let $|b_n=W_{3m-2}(n)|$ for $n\geq 1$. It is then easy
to see that $b_1=3m-2$, and 
\begin{equation} \label{bnrec} b_n=(3m-2)b_{n-1}-(m-1)b_{n-2} \end{equation}
for $n\geq 2$. Indeed, if we take an element of $V_{3m-2}(n-1)$, and append
one of our $3m-2$ letters to its end, we will get an element 
of $W_{3m-2}(n)$, except in the $(m-1)b_{n-2}$ cases in which the last two 
letters of the new word form one of the forbidden factors.

Introducting the generating function $B(x)=\sum_{n\geq 1}b_nx^n$, we can turn
formula (\ref{bnrec}) into a functional equation. Solving that equation,
we get
\begin{equation} \label{funcb} B(x)=\frac{1}{1-(3m-2)x+(m-1)x^2}.\end{equation}

Finding the roots $r_1=\frac{3m-2+\sqrt{9m^2-16m+8}}{2(m-1)}$  and $r_2=\frac{3m-2-\sqrt{9m^2-16m+8}}{2(m-1)}$ of the denominator of $B(x)$, we see that $B(x)$ can
be converted to the partial fraction form 
\[B(x)=\frac{C_1(m)}{1-\frac{x}{\beta_1}} + \frac{C_2(m)}{1-\frac{x}{\beta_2}},\]
and our claim is now routine to prove. 
\end{proof}

\begin{corollary} \label{evencor}
For all {\em even} positive integers $k$, the inequality
\[S_n(M_{k}) \leq (2.25k^2)^n\]
holds.
\end{corollary}

\begin{proof} Let $k=2m$.
Part (b) of Lemma \ref{inducl} inductively constructs an injective map $f_{2m} :Av_{n}(M_{2m})
\rightarrow W_{3m-2}(n) \times W_{3m-2}(n)$. That map is not bijective. Indeed, it is obvious
from the definition of $f_{2m}$ that if $f_{2m}(p)=(w,w')$, then both $w$ and $w'$ must start
with the letter 1.

Therefore, we know that 
\begin{equation} \label{almost}
S_n(M_{k})\leq |W_{3m-2}(n-1)|^2=\left(C_1(m)\beta_1^{n-1} + C_2(m)\beta_2^{n-1}\right)^2.
\end{equation}

It is routine to verify that for all  integers $m>1$, the inequality
$\beta_2<1$ holds.  As $\beta_1\beta_2=m-1$, this means that $\beta_1>m-1$, and so
$C_1=\frac{\beta_1}{\beta_1-\beta_2}$, this implies the inequality $C_1<2$ for $m\geq 3$. 
This same inequality can be verified for $m=2$, since in that case, $\beta_1=2+\sqrt{3}$, and
$\beta_2=2-\sqrt{3}$. Furthermore, $C_2=\frac{\beta_2}{\beta_2-\beta_1}<0$ since the 
denominator is negative. Hence,(\ref{almost}) implies
\[S_n(M_k) \leq (2\beta_1^{n-1})^2 \leq  \beta^{2n} .\]

It is easy to prove from the definition of $\beta$ that for all integers $m>1$, the inequality
 $\beta < 3m-2$ holds. Therefore, 
\[S_n(M_k)\leq \beta^{2n} < (3m-2)^{2n}= (1.5k-2)^{2n}=(2.25k^2-3k+4)^n.\]
\end{proof}

\begin{proposition}
\label{upperv} For all integers $m\geq 2$, 
we have
\[|V_{3m-2}(n)| = K_1 \cdot \alpha_1^n  + K_2 \cdot \alpha_2^n \]
where $\alpha_1 = \frac{3m-4+\sqrt{9m^2-28+24}}{2}$,  and
 $\alpha_2 = \frac{3m-4-\sqrt{9m^2-28+24}}{2}$, while $K_1=\frac{\alpha_1}{\alpha_1-
 \alpha_2}$ and $K_2=\frac{\alpha_2}{\alpha_2-\alpha_1}$. 
\end{proposition}

\begin{proof} Analogous to that of Proposition \ref{upperw}.
\end{proof}

\begin{corollary}
For all {\em odd} positive integers $k$, the inequality 
\[S_n(M_{k}) \leq (2.25k^2)^n\]
holds.
\end{corollary}

\begin{proof} The proof is analogous to that of Corollary \ref{evencor}.
The only difference is that now we set $k=2m-1$, and then we use part (a)
of Lemma \ref{inducl}. We get that 
\[S_n(M_{k})\leq |V_{3m-4}(n-1)|^2\leq \alpha_1^{2n}.\]
So \begin{eqnarray*} S_n(M_k) & \leq  & \alpha_1^{2n} \\
&  <  & (3m-4)^{2n} \\
& = & (1.5k-2.5)^{2n}\\
& = & (2.25k^2-7.5k+6.25)^n \\
& \leq & (2.25k^2)^n.\end{eqnarray*}
\end{proof}

\end{document}